\documentclass{amsart}

\RequirePackage{hyperref}

\usepackage{amsmath,amsfonts,tikz,tikz-3dplot,capt-of,hyperref,capt-of, enumitem}
\usetikzlibrary{positioning}
\usepackage{graphicx}
\newtheorem{lem}{Lemma}
\newtheorem{thm}[lem]{Theorem}

\newtheorem{prob}[lem]{Problem}
\newtheorem{prop}[lem]{Proposition}
\newtheorem{defn}[lem]{Definition}
\newtheorem{rmk}[lem]{Remark}

\newtheorem{ex}[lem]{Example}
\newtheorem{conv}[lem]{Convention}
\newcommand{\PP}{\mathbb{P}}
\newcommand{\CC}{\mathbb{C}}
\newcommand{\ZZ}{\mathbb{Z}}
\newcommand{\FF}{\mathbb{F}}
\usetikzlibrary{arrows.meta, shadows, fadings,shapes.arrows,positioning}
\tikzset{My Arrow Style/.style={single arrow, fill=red!40, anchor=base, align=center,text width=2.8cm}}

\usetikzlibrary{shapes.geometric,arrows.meta}
\tikzset{
    auto,node distance =1 cm and 1 cm,semithick,
    state/.style ={ellipse, draw, minimum width = 0.7 cm},
    bidirected/.style={Latex-Latex,dashed},
    connected/.style={dashed,-},
}

\numberwithin{lem}{section}

\begin{document}

\title[Local Calabi--Yau 3-folds for some shrinkable surfaces]{Local Calabi--Yau 3-folds for some rank 2 shrinkable surfaces}
\author{Sungwoo Nam}
\address{Department of Physics, POSTECH, Pohang 790-784, Korea}
\email{sungwoo2@postech.ac.kr}

\begin{abstract}
Motivated by 5d rank 2 SCFTs, we construct a smooth, non-compact Calabi--Yau 3-fold $X$ containing a rank 2 shrinkable surface $S=S_1\cup S_2$ glued over a smooth curve. This construction will be a generalization of the construction of a local surface for a smooth surface $S$.
\end{abstract}

\maketitle

\section{Introduction}
Given a projective surface $S$, it is often desirable to have a (not necessarily compact) Calabi--Yau (CY) 3-fold $X$ containing $S$ for various reasons. When we can take $X$ to be a small neighborhood of $S$ in any CY 3-fold in an appropriate sense, such $X$ would be called a \emph{local CY 3-fold of $S$}. Mathematically, they are studied extensively from the points of view of mirror symmetry, enumerative geometry, and canonical 3-fold singularities \cite{Hosono,LocMi,XieYau}. Such local CY 3-folds appear in many contexts of physics as well, such as M-theory and the five-dimensional superconformal field theories (SCFTs) \cite{JKKV, Mosei}. For a smooth surface $S$, there is a well-defined notion of \emph{the local surface of $S$} and it is defined by Tot$(\omega_S)$ where $\omega_S$ is the canonical line bundle of $S$. This is exactly a CY 3-fold containing $S$ via embedding by the zero section of $\omega_S.$ For singular surfaces, this construction no longer gives us a smooth CY 3-fold. This is because, even when there is a dualizing line bundle $\omega_S$, the total space of this line bundle is singular.

In this paper, we study this problem of constructing local CY 3-folds in the context of 3-fold canonical singularities, motivated by 5d SCFTs. Recall that 3-fold canonical singularities are not necessarily isolated. Moreover, if they are not isolated, then they are du Val singularities in codimension at least 2. That is, analytically locally it is isomorphic to 
\begin{equation}
    \text{du Val singularity}\times\mathbb{A}^{1}
\end{equation}
around a general point of any codimension 2 strata.
See, for instance \cite{reid}. These 3-fold canonical singularities are closely related to the physics of 5d SCFTs as studied in \cite{XieYau}. 

However, the classification of 3-fold canonical singularities is notoriously hard. Instead, in \cite{JKKV}, authors used the notion of \emph{shrinkable surfaces} to study 5d SCFTs. Up to a notion called \emph{physical equivalence}, they provide conjectural geometric classifications of 5d SCFTs in terms of shrinkable surfaces. Although we do not discuss physical equivalence in detail in this paper, roughly speaking, it relates different surfaces that give rise to same physics and in particular, same local GW theory (see for instance Remark~\ref{cau}). Authors of \cite{JKKV} also conjectured that shrinkable surfaces can be characterized using the intersection theory of an ambient CY 3-fold. Motivated by this, we define the notion of (pre-)shrinkable surfaces (see also \cite{preprint}) using the following definition of shrinkable CY 3-folds.

\begin{defn}\label{shr}
    Let $S=\cup^n_iS_i$ be a projective simple normal crossing (snc) surface, i.e., a projective surface with each $S_i$ smooth projective and $X$ be a smooth Calabi--Yau 3-fold containing $S$. Then $X$ is called shrinkable of rank $n$ if the following three conditions hold.  Let $J=\sum_{i=1}^n a_i[S_i]$ be an integral divisor class of a Calabi--Yau 3-fold $X$ containing $S$.
    \begin{enumerate}[label=(\roman*).]\label{defsh}
        \item There are integers $a_i\geq0$ such that \begin{equation*}\label{usingcy}
        -C\cdot J\geq0
    \end{equation*} for any curve $C$ on $S$. The notation $\cdot$ denotes an intersection product on $X$. 
    \item For each $i$, 
    \begin{equation*}\label{eq1}
        J^2S_i\geq0.
    \end{equation*} 
   \item For at least one $i$, \begin{equation*}\label{eq2}
       J^2S_i>0.
   \end{equation*}
    \end{enumerate} 
    
    \end{defn}
    The intersection numbers $-J\cdot C$ and $J^2S_i$ in the above definitions are what physicists call volumes of 2-cycles and 4-cycles, respectively.
    Observe that the intersection products in the above condition are independent of a choice of a smooth CY 3-fold $X$ containing $S$ and it can be computed intrinsically (without using $X$) on $S$ as follows. When $C$ is an irreducible curve
    \begin{align}\label{intersection1}
        C\cdot S_i=\text{deg}(\omega_{S_i}\vert_C) &\text{ if }C\subset S_i\\\label{intersection2}
        C\cdot S_j=(C\cdot (S_i\cap S_j))_{S_i} &\text{ if }C\not\subset S_j
    \end{align}
    and we can extend linearly. Similarly, the triple intersection formulas in ~(\ref{eq1}) and ~(\ref{eq2}) can be computed intrinsically on $S$. 
    
    To formulate an intrinsic notion of a shrinkable surface, we impose one necessary condition for an snc surface to be embeddable in a shrinkable 3-fold. The following condition is called \emph{the Calabi--Yau condition}. For each $i\neq j$, we require
\begin{equation}\label{CY}
    (C_{ij})^2_{S_i}+(C_{ij})^2_{S_j}=2g(C_{ij})-2.
\end{equation}
for $C_{ij}=S_i\cap S_j$ where self-intersection is computed in each component $S_i$ and $S_j$. Here $g(C_{ij})$ denotes the arithmetic genus of $C_{ij}.$ It is elementary to see that this is a necessary condition for $S$ to be in a CY 3-fold.

    We now define a main character of this paper.
    
    \begin{defn}\label{preshrink}
        Let $S=\cup_i^nS_i$ be an snc surface. We say $S$ is \emph{pre-shrinkable} of rank $n$ if it satisfies the Calabi--Yau condition~(\ref{CY}) and conditions in (i)-(iii) in Definition~\ref{shr}.

    We then define $S$ to be \emph{shrinkable} if it is pre-shrinkable and if there exists a smooth CY 3-fold containing $S$.

    \end{defn}
   \begin{rmk}
       In fact, we can weaken the condition (iii) of Definition~\ref{shr} to allow examples with $J^2S_i=0$ for all $i$ while satisfying (i). There exist such surfaces, for instance, $\FF_0\cup\FF_{10}$. These geometries are expected to come from 6d theory. As we will focus on surfaces from 5d SCFTs, we stick to Definition~\ref{preshrink}.
   \end{rmk}

\begin{ex}
    If $S$ is a smooth projective surface, this condition reduces to deg$(\omega_S\vert_C)\leq0$ for any curve $C$ on $S$ and $K_S^2>0$. These are the surfaces whose canonical bundle $\omega_S$ is nef with $K_S^2>0$.

    For examples of rank 2 pre-shrinkable surfaces, see Section ~\ref{classification}.
\end{ex}

As we can see from this example, shrinkable surfaces can be seen as generalizations of contractible surfaces in a CY 3-fold. For instance, del Pezzo surfaces are shrinkable and well-known to be contractible on their local surface. The Hirzebruch surface $\FF_2$ is also shrinkable and it can be contracted to a point together with a non-compact divisor, see Example~\ref{f2sh}. In general, shrinkable surfaces are expected to contract to unions of points and non-compact curves, engineering 3-fold canonical singularities \cite{XieYau}.

In \cite{JKKV}, physics associated with rank 2 pre-shrinkable surfaces was studied. However, for rank $>1$ pre-shrinkable surfaces, we do not in general know whether every pre-shrinkable surface is a shrinkable surface. This motivates the following problem, which we call the embedding problem.

\begin{prob}\label{embedding}
   Is every pre-shrinkable surface of rank $n>1$ shrinkable?
\end{prob} 
In physics, for rank 2 surfaces in \cite{JKKV}, a smooth CY 3-fold can be obtained from k\"ahler deformations of elliptic CY 3-folds containing pre-shrinkable surfaces, so we expect a positive answer. The goal of this paper is to provide mathematical justification for this for rank 2 surfaces.

Since shrinkable CY 3-folds are used to engineer 5d SCFTs via M-theory, there are  BPS states. Such BPS states are generated by M2 branes wrapped on holomorphic curves and are thus expected to be related to curve-counting theories. The local Gromov--Witten (GW) theory of shrinkable surfaces was proposed in \cite{preprint} to build this connection. In particular, it has a well-defined meaning as a contribution to the 3-fold invariants assuming a positive answer to Problem~\ref{embedding}. 

For the remaining part of this paper, we write $S_1\cup_C S_2$ for a rank 2 surface where $C=S_1\cap S_2$ is the double curve, which we will assume to be \emph{smooth}. We suppress $C$ from the notation if it is clear from $S_1$ and $S_2$. The following is the main construction of this paper.

\begin{thm}
    Let $S$ be a pre-shrinkable surface $S=S_1\cup_C S_2$, $C\simeq \PP^1$. Then $S$ is embeddable in a smooth, non-compact CY 3-fold if
    \begin{enumerate}
        \item $S_2\simeq \FF_n$ with $C$ being the $(-n)$ curve for $n>0$, $S_2\simeq \FF_0$ with $C$ being a ruling or $S_2\simeq \FF_1$ with $C$ a section of $\FF_1\to\PP^1$ with $C^2=1$.
        \item $S_2\simeq\PP^2$ and $C$ is a line in $\PP^2$.
    \end{enumerate}
\end{thm}
This theorem covers 61 rank 2 pre-shrinkable surfaces out of all 64 appearing in \cite{JKKV} up to physical equivalence\footnote{These 64 surfaces do not cover all rank 2 5d SCFTs. In \cite{bootstr}, it was pointed out that there are 5d SCFTs that do not yet have geometric descriptions using shrinkable 3-folds. For rank 2, two such theories are called local $\PP^2\cup\FF_6+``1\textbf{Sym}"$ and $SU(3)_8$. The classification of 5d SCFTs is not yet complete. In this paper we only focus on SCFTs that have geometric description.}. Together with a toric construction, 62 pre-shrinkable surfaces are indeed proven to be shrinkable. For the other two pre-shrinkable surfaces, our result can be applied to construct non-pre-shrinkable surfaces that are physically equivalent to those surfaces (see Remark~\ref{cau}). The novelty of this theorem is that it can be applied to non-toric surfaces such as $\PP^2\cup\FF_6$ or $\FF_0\cup\FF_6$. See Section~\ref{classification}. 

Our proof uses the construction of local surfaces for \emph{orbifolds}. Using orbifolds, we can engineer transverse $A_n$-singularities on a curve on a surface in a CY 3-fold.

Having constructed a CY 3-fold containing $S$, we can define the local GW invariants of $S$ if the moduli spaces of stable maps to $S$ and $X$ are identical as Deligne-Mumford stacks. The definition, together with an interpretation of these invariants, is discussed in the companion paper \cite{preprint}.

For interesting examples, we examine some toric constructions, the weighted projective stacks of dimension 2. They can be contracted to a 3-fold cyclic quotient singularities which are extensively studied under the names of $\CC^3/G$ orbifolds and $G$-Hilbert schemes. The root construction on some weighted projective planes can be easily understood using toric diagrams \cite{toricdm}.
Moreover, it has at worst cyclic quotient singularities. As such, it provides many examples of the main theorems. From the SCFT point of view, it generates many higher-rank examples (See \cite{orbifold}, Section 3.1.1). Also there are plenty of examples of transverse singularities with \emph{dissident points}. These cannot just be obtained by taking a root stack or a canonical stack (introduced in Section~\ref{dmorbi}) to an underlying projective surface but a combination of these operations. Understanding these examples in terms of combining the two stack-theoretic operations and generalizing to non-toric examples might provide more solutions to Problem~\ref{embedding}.

\smallskip\noindent
{\bf Notations.} Throughout, we use the following notations. For $\FF_0$, we denote the two rulings by $f_1$ and $f_2$. If $n>0$, then we use $e$ to denote the curve $e^2=-n$, and $f$ is the fiber class of $\FF_n$. For $\PP^2$, $\ell$ is the class of a line. Lastly, $dP_n$ and Bl$_nS$ denote the blow-ups of $\PP^2$ and $S$ respectively, at $n$-general points, and $x_i$ denotes the exceptional $(-1)$ curve from blowing up $i$th point.

\smallskip\noindent
{\bf Acknowledgements.}  We would like to thank Sheldon Katz for his encouragement to write this paper and numerous helpful discussions. We would also like to thank Hee-Cheol Kim for physical intuition on 5d SCFTs and for clarifying the contents of \cite{bootstr} and \cite{JKKV}.

\section{DM stacks and canonical bundles}\label{dmorbi}
One ingredient for our local Calabi--Yau construction is the notion of stacks or orbifolds. We refer to \cite{sasaki} Chapter 4 or \cite{rossthomas} for foundational materials on orbifolds and orbifold line bundles. Here we briefly review relevant notions to fix the notations. In this paper, we only consider stacks with the following assumptions. This is mostly to ensure the existence of a canonical (dualizing) line bundle.

\begin{conv}
    By a DM stack, we mean a smooth, separated, tame Deligne-Mumford stack of finite type over $\CC$. By an orbifold, we mean a DM stack with a generically trivial stabilizer. 

    We further assume that the stabilizer groups of points on a DM stack are cyclic and the coarse moduli space $X$ of a DM stack $\mathcal{X}$ is a quasiprojective variety.
\end{conv}

Based on the results in \cite{bottom}, such DM stacks or orbifolds can be constructed starting from a variety with tame quotient singularities. There are two operations to get a smooth DM stack from such a variety. For the explicit definition of these constructions, see \cite{bottom}.

\begin{enumerate}
    \item The canonical stack morphism $\mathcal{X}^{\mathrm{can}}\to \mathcal{X}$. This is associated with a smooth DM stack $\mathcal{X}$ with the coarse moduli space morphism $\pi:\mathcal{X}\to X$ being an isomorphism away from codimension 2 loci. This is the terminal object in the category of orbifolds with dominant, codimension-preserving morphism to $X$.

    \item The $n$th root stack morphism $\sqrt[n]{\mathcal{D}/\mathcal{X}}\to\mathcal{X}$ along an effective Cartier divisor $\mathcal{D}$ of a smooth DM stack $\mathcal{X}$. We omit $n$ from the notation if $n=2.$
\end{enumerate}
Although these two operations can be defined for more general DM stacks, for our application we focus on the cases in which we have either of the following cases
\begin{enumerate}
    \item where $\mathcal{X}$ is a projective surface with finite quotient singularities so that $\mathcal{X}^{\mathrm{can}}$ exists and is smooth as a DM stack.
    \item where $\mathcal{X}$ is a smooth projective surface and $\mathcal{D}=\sum D_i$ is a normal crossing divisor where each $D_i$ is a smooth divisor so that the root stack $\sqrt[n]{\mathcal{D}/\mathcal{X}}$ is smooth as a DM stack.  
\end{enumerate}

An orbifold can be defined using charts and atlas like smooth manifolds. An orbifold chart for an $n$-dimensional orbifold $\mathcal{X}$ is given by $(\widetilde{U},\Gamma,\varphi)$ such that 
\begin{itemize}
\item an open connected set $\widetilde{U}\simeq \CC^n$ containing the origin
\item a finite group $\Gamma$ (called a uniformizing group) acting effectively on $\widetilde{U}$
    \item an $\Gamma$-invariant map $\varphi:\widetilde{U}\to U\subset\mathcal{X}$ which induces an homeomorphism $\widetilde{U}/\Gamma\simeq U$ to an open subset $U$ of $\mathcal{X}.$
\end{itemize}
The notion of an atlas and a refinement of an atlas generalize straightforwardly to orbifolds. See \cite{sasaki}, Chapter 4 for details.

An orbifold line bundle is locally defined on a chart $(\widetilde{U},\Gamma_i,\varphi_i)$ by a fiber bundle $E_{\widetilde{U}_i}$ with a fiber $G$-representation $\CC$ with a homomorphism $h_i:\Gamma_i\to G$ such that
\begin{itemize}
    \item  if $b$ lies in the fiber over $\widetilde{x}_i\in\widetilde{U}_i$ then $h_i(\gamma)\cdot b$ lies in the fiber of $\gamma^{-1}\cdot\widetilde{x}_i$

    \item if $\lambda_{ji}:\widetilde{U}_i\to \widetilde{U}_j$ is an embedding, there is a transition map 
    \begin{equation}
        \lambda_{ji}^*:E\vert_{\widetilde{U}_j\vert_{\lambda_{ij}(\widetilde{U_{i}})}}\to E\vert_{\widetilde{U}_i}
    \end{equation}
    that satisfies the following condition. If $\gamma\in\Gamma_i$ and $\gamma'\in\Gamma_j$ is the unique element satisfying $\lambda_{ji}\circ \gamma=\gamma'\circ\lambda_{ji}$ then $h_{\widetilde{U}_i}(\gamma)\circ\lambda_{ji}^*=\lambda_{ji}^*\circ h_{\widetilde{U}_j}(\gamma').$
    \item If $\lambda_{kj}:\widetilde{U}_j\to\widetilde{U}_k$ is another embedding then $(\lambda_{kj}\circ\lambda_{ji})^*=\lambda_{ji}^*\circ\lambda_{kj}^*.$
\end{itemize}

In particular, each fiber can have non-trivial group action by a uniformizing group. 

A section $s$ of an orbifold line bundle $E$ on $\mathcal{X}$ can be defined by the following data. 

\begin{itemize}
    \item Over $(\widetilde{U}_i,\Gamma_i,\varphi)$, let $\widetilde{x}\in \widetilde{U}_i$. For each $\gamma\in \Gamma_i$ 
    \begin{equation}
        s_i(\gamma^{-1}\widetilde{x})=h_{\widetilde{U}_i}(\gamma)\cdot s_i(\widetilde{x})
    \end{equation}
    \item If $\lambda_{ji}:\widetilde{U}_i\to \widetilde{U}_j$ is an embedding then 
    \begin{equation}
        \lambda_{ji}^*s_j(\lambda_{ji}(\widetilde{x}))=s_i(\widetilde{x}).
    \end{equation}
\end{itemize}

Note that the zero section is always a section of any orbifold line bundle.

The total space $E$ (abusing notation, we use the same letter to denote the bundle and its total space) of an orbifold line bundle has an orbifold structure induced from that of the base. Consider an orbifold chart $(\widetilde{U}_i,\Gamma_i,\varphi_i)$ which is also trivializing $E$. Then we define an orbifold structure on the total space by charts of the form $(\CC\times \widetilde{U}_i,\Gamma'_i,\varphi'_i)$. Here we extend the action of $\Gamma_i$ to $\CC\times \widetilde{U}_i$ by $\gamma\cdot(b,\widetilde{x}_i)=(h_i(\gamma)\cdot b,\gamma^{-1}\widetilde{x})$. Then $\Gamma_i'$ is the subgroup of $\Gamma_i$ that stabilizes ($b,\widetilde{x}_i).$ 

The notion of divisors also generalizes. It is convenient to work with $\mathbb{Q}$-divisors. For complex orbifolds, every Weil divisor has a multiple, which is Cartier. A Baily divisor is a $\mathbb{Q}$-Weil divisor whose inverse image in every uniformizing chart $\varphi:\widetilde{U}\to U$ is Cartier with some natural compatibility conditions. For a Baily divisor $D$, we can also define its associated line bundle $\mathcal{O}(D)$. For us, it is enough to note that the orbifold canonical divisor $K^{orb}_{\mathcal{X}}$ is a Baily divisor and its associated line bundle is the orbifold canonical line bundle $\omega_{\mathcal{X}}$. It is also characterized as a sheaf of holomorphic differential forms which is useful for local calculations. 

Of particular interest to us, for orbifolds considered in this paper, there exists an orbifold canonical line bundle. The orbifold canonical bundle is not just a pullback of the canonical bundle of the coarse moduli space. It receives corrections from the ramification loci.

\begin{prop}(\cite{sasaki})
   Let $\pi:\mathcal{X}\to X$ is the coarse moduli space map ramified over divisors $D_i$ with index $m_i$. Then the orbifold canonical divisor is given by
   \begin{equation}
       K_\mathcal{X}^{orb}=\pi^*K_X\oplus \bigoplus_{i}\left(1-\frac{1}{m_i}\right)D_i.
   \end{equation}
\end{prop}

With this canonical bundle, we can consider the total space of the canonical bundle. The total space of an orbifold line bundle also inherits an orbifold structure from the base. In particular, a fiber inherits an action of the stabilizer of the base. This way, the coarse moduli space of the total space of the canonical bundle can have 3-fold singularities.

Among all orbifolds, orbifold Calabi--Yau 3-folds are the most important to us.
\begin{defn}
    Let $\mathcal{X}$ be an orbifold. If $\omega_{\mathcal{X}}\simeq\mathcal{O}_{\mathcal{X}}$, then $\mathcal{X}$ is called an orbifold Calabi--Yau 3-fold.
\end{defn}

This definition implies that the local model around a point $p$ is $[\CC^n/G]$, $G\subset SL(n,\CC).$

We will use the following example over and over again in this paper.
\begin{ex}
    Let $\mathcal{S}$ be an $n$-dimensional orbifold for an integer $n>0$. Then the total space of the orbifold canonical bundle $\mathcal{X}=$Tot $(\omega_{\mathcal{S}})$ is an orbifold Calabi--Yau $(n+1)$-fold. The orbifold canonical bundle of $\mathcal{X}$ is trivial. The very construction tells us that $\omega_{\mathcal{X}}$ is locally trivial, that is, the fiber is just $\CC$, with no nontrivial action of a finite group coming from the stabilizer group of the base. Then vanishing first Chern class can be seen, just like complex manifolds, which implies $\omega_{\mathcal{X}}\simeq\mathcal{O}_{\mathcal{X}}$.
\end{ex}

\section{Embedding rank 2 surfaces to a smooth CY 3-fold}
Using the machines of the previous section, we now construct a smooth CY 3-fold containing a rank 2 pre-shrinkable surface $S$. Readers can easily use the construction given in this section to produce CY 3-folds containing surfaces $S=S_1\cup_C S_2$ that are not necessarily pre-shrinkable, but satisfy the Calabi--Yau condition. Note that it is also possible to take $g(C)>0.$ We restrict to $g(C)=0$ case as it is the most interesting for 5d SCFTs.
 
 Using 2-dimensional orbifolds, we can embed almost all pre-shrinkable surfaces showing up in \cite{JKKV} into a smooth CY 3-fold. The idea is to take the total space of the canonical bundle $\omega_{\mathcal{S}}$ of an orbifold $\mathcal{S}$. Depending on the orbifold structure, the total space has either codimension 3 or codimension 2 singularities. If the singularities are of codimension 2, often they are of nice form, for instance, transverse $A_n$ singularities. In the case of transverse $A_n$ singularities, we know that such singularities allow a crepant resolution by successively blowing up singular loci. Moreover, such a crepant resolution is unique (See section 4 of \cite{oade}). For transverse $A_1$ case, the exceptional divisor is a ruled surface over $C$. Often, the Calabi--Yau condition determines the isomorphism type of this ruled surface. This is mostly the case when $C\simeq\PP^1.$

\begin{lem}\label{root}
    Let $S$ be a smooth projective surface and $D$ be a smooth divisor with self-intersection $D^2=d$ and arithmetic genus (viewed as a smooth curve) $g$. Then the total space of the canonical bundle of $\sqrt[n]{D/S}$ is an orbifold CY 3-fold $\mathcal{X}$ whose coarse moduli space $X$ contains $S$ as the image of the zero section. The only singularities of $X$ are transverse $A_n$ singularities along $D$ in the zero section.
\end{lem}
\begin{proof}
    Since $D$ is smooth, the root stack has a coarse moduli space isomorphic to $S$. This comes from local $\sqrt[n]{D/S}$. Above a stacky point, the fiber has the induced $\mu_n$-action. Because of this $\mu_n$ action, the singularities are localized in the zero section. Using a local calculation, it follows that they form transverse $A_n$ singularities along $D$ inside the zero section.
\end{proof}

\begin{thm}\label{main}
    Let $S=S_1\cup_C S_2$ be a pre-shrinkable surface. If $C\simeq\PP^1$ is the $(-n)$ curve with $n>0$ on $S_2\simeq\mathbb{F}_n$ or $C$ is a ruling of $\FF_0$, then $S$ can be embedded into a smooth non-compact CY 3-fold. 

    In addition, if $C\simeq \PP^1$ is a $(-3)$ curve on $S_2$, then we have $S_1\cup\FF_1,$ and $C$ is in class $e+f$ in $S_1.$
    
    Particularly, 52 pre-shrinkable surfaces of these forms out of all 64 surfaces in Table \ref{sh} and Table \ref{fig:table2} are shrinkable according to this theorem.
\end{thm}
\begin{proof}
    Apply the square-root stack construction for $C\subset S_1$ and set $\mathcal{S}=\sqrt{C/S_1}.$ The coarse moduli space of the total space of the canonical bundle of $\mathcal{S}$ is a singular CY 3-fold that has transverse $A_1$ singularities along $C$ in the zero section. By blowing up this singular locus, we get the exceptional divisor $S_2$ ruled over $C$. The isomorphism type of $S_2$ is determined by the Calabi--Yau condition on $C$. Since we still have $S_1$, we get $S_1\cup_C S_2$ in the resolved 3-fold. Observe that as this is a small resolution, it is creapnt, and thus the resolved 3-fold is still Calabi--Yau.

    The case for $C$ being a $(-3)$ curve in $S_2$ follows from the fact that $C$ must be a section of a ruled surface $S_1$ over $\PP^1$ with self-intersection 1 in $S_1.$
\end{proof}

In particular, if we have transverse $A_1$ singularities, a single blowup along the singular locus gives a smooth CY 3-fold and the exceptional divisor is a $\PP^1$-bundle over the singular locus. This construction, when applied to $\PP^2$ along a smooth conic $C$, gives $\PP^2\cup\FF_6$. Similarly, we can construct $\FF_0\cup\FF_4$ where the double curve is $f_1+f_2$ in $\FF_0$, the sum of two rulings of $\FF_0$. Note that these examples are not toric as the double curve $C$ is not a torus-invariant divisor in at least one component. Most surfaces in Table~\ref{sh} (except $dP_2\cup dP_2$ and Bl$_3\FF_3\cup\PP^2$) are embeddable using this theorem.

\begin{rmk}\label{cau}\begin{enumerate}
    \item 
    In the above theorem, if $(C^2)_{S_1}<-3$, then there is an ambiguity in determining the exceptional divisor using only the self-intersection of $C$. However, exceptional divisors of the transverse $A_1$ singularities can still be computed  in many cases. However, they do not necessarily give pre-shrinkable surfaces. For instance, when we apply our construction to $\sqrt{e/\FF_5}$, we get a CY 3-fold containing $\FF_5\cup\FF_3$, instead of pre-shrinkable $\FF_5\cup\FF_1$. 

\item In Section 3 of \cite{JKKV}, the Hanany-Witten (HW) transition was discussed. This is another way to get another physically equivalent surface from the other. The physics of local $\FF_2$ and local $\FF_0$ are equivalent up to decoupled free sectors (mathematically, local GW invariant is not well-defined for $(-2)$ curve). This can be understood as meaning that after throwing some curve classes away, the data in their local GW theory are the same. HW transition can be performed in higher rank cases, giving physical equivalence between $\FF_2\cup\FF_4$ and $\FF_0\cup\FF_4$ where the double curve is the $(-4)$ curve on $\FF_4$ for both surfaces.

In particular, this gives an embedding of surfaces Bl$_6\FF_4\cup\FF_2$ (we identify $e$ on Bl$_6\FF_4$ and a (+2) curve on $\FF_2$) when we apply the square-root stack construction to $\sqrt{e/\text{Bl}_6\FF_4}$. This geometry is physically equivalent to Bl$_6\FF_4\cup\FF_0$. Note that in this case, both surfaces are pre-shrinkable.

    \item It is possible to generalize Theorem~\ref{main} to obtain higher rank surfaces by considering $n$th-root stacks instead of square-root stacks. However, for transverse $A_n$ singularities with $n>1$, we also need to understand the monodromy to fully understand the geometry of the exceptional divisors (See Sections 3 and 4 of \cite{oade}). We leave studying this geometry for future work.
\end{enumerate}
\end{rmk}

To state the next theorem, we use the multiplicative notation for the cyclic group $\mu_n$ of order $n$ and will think of them as a subgroup of $\CC^*$ embedded as the group of $n$th roots of unity.

\begin{lem}\label{iso}
    Let $S$ be a normal projective surface with cyclic quotient singularities at finitely many points $x_1,\dots,x_n$. Let $\frac{1}{r_i}(1,q_i)$ be its singularity type at $x_i$ with $q_i$ being coprime to $\leq r_i$. That is, the singularity is isolated and locally given by the quotient
    \begin{equation}
        \zeta\cdot(z_1,z_2)=(\zeta z_1,\zeta^{q_i}z_2)\text{ where }\zeta=e^{\frac{2\pi i}{r_i}}. 
    \end{equation}Then the total space of the canonical bundle of $\mathcal{S}=S^{\mathrm{can}}$ has the coarse moduli space $X$ which has
    \begin{enumerate}
        \item a 3-fold cyclic quotient singularity at $x_i$ inside the zero section if $1+q_i\neq r_i$
        \item transverse $A_{r_i-1}$ singularities if $1+q_i=r_i$ along the fiber of $x_i$. 
    \end{enumerate}
    In either case, successive blow-ups of singular loci resolve the singularities. The resolved 3-fold then contains the minimal resolution $\widetilde{S}$ of $S$ together with the compact exceptional divisors from 3-fold cyclic quotient singularities of type (1) and some non-compact divisors either from type (1) or (2).
\end{lem}

\begin{proof}
    Since the singularities on $S$ are isolated, we can work locally. The differential forms have induced action from the base, which makes the total space a CY 3-fold. If $1+q_i\neq r_i$ then there is a nontrivial action of $\mu_{r_i}$ on the fiber of the canonical bundle over $x_i$. Thus, the singularity is localized in the zero section.

    If $1+q_i=r_i$, the surface $S$ has the surface $A_{r_i-1}$ singularity at $x_i$. Then the fiber of the canonical bundle at $x_i$ has only trivial action, so we have the same singularity along the entire fiber, which gives transverse $A_{r_i-1}$ singularities along the fiber.
\end{proof}

We will see in the next section examples of these surfaces. The following is a simple consequence of the above lemma for pre-shrinkable surfaces which is in fact true for all surfaces satisfying the Calabi--Yau condition. In particular, this example tells us how to embed surfaces like Bl$_8\FF_3\cup \PP^2$ in a CY 3-fold.

\begin{thm}\label{mainthm2}
    Let $S=S_1\cup_CS_2$ be a surface satisfying the Calabi--Yau condition. Suppose $S_2\simeq \PP^2$ and $C$ is a line on $S_2$. Then $S$ can be embedded into a smooth, non-compact CY 3-fold.
\end{thm}
\begin{proof}
    Note that by the CY condition, $(C\cdot C)_{S_1}=-3$. Let $S\to S'$ be a contraction of this $(-3)$-curve. Then $S'$ has $\frac{1}{3}(1,1)$ singularity. Since $S'$ has only a quotient singularity, there exists a canonical DM stack $(S')^\text{can}$. The coarse moduli space of the total space of the canonical bundle, $\text{Tot}(\omega_{(S')^\text{can}})$, then has an isolated singularity of type $\frac{1}{3}(1,1,1)$. This quotient singularity can be resolved with the exceptional divisor $\PP^2$. Since the resolution is just done by a single blowup of the singular point, on $S'$, it is restricted to a single blowup of the singular point, which gives back a smooth rational $(-3)$ curve. That is, this blow up restricts to $S\to S'$. By the CY condition, this must be identified with a line in $\PP^2.$
\end{proof}

Ideally, one might hope to combine the above two constructions. For a smooth DM stack $\mathcal{X}$, we have the following factorization
\begin{equation}
    \mathcal{X}\simeq (\sqrt{\mathcal{D}/X^{\mathrm{can}}})^{\mathrm{can}}\to \sqrt{\mathcal{D}/X^{\mathrm{can}}} \to X^{\mathrm{can}}\to X
\end{equation}
where $\mathcal{D}$ is the associated Cartier divisor of $D$ in $X^{\mathrm{can}}.$ As a result, the triple $(X,D_i,e_i)$ determines $\mathcal{X}$. However, as it is observed in \cite{bottom}, not every triple determines a smooth stack. In particular, they gave an example where $X\simeq\mathbb{A}^3$ and $D=V(xy+z^2)$, which is singular, then $\sqrt{D/X}$ has a hypersurface (conifold) singularity that is not a quotient singularity. Then the canonical stack does not exist. Therefore we can only combine the two constructions when they arise as a factorization of a smooth DM stack. In general, the resulting singular 3-fold has generically transverse $A_n$ singularities with a finite number of dissident points. See Example~\ref{z7}.

\section{Application: Shrinkability of pre-shrinkable surfaces}\label{classification}
Using the techniques discussed so far and a toric construction for $dP_2\cup dP_2$ below, we can embed 62 pre-shrinkable surfaces (not physically equivalent to each other) in a shrinkable CY 3-fold. In particular, we can embed non-toric surfaces like $\PP^2\cup\FF_6$ (apply Theorem~\ref{root} to a smooth conic in $\PP^2)$ and Bl$_5\FF_3\cup\PP^2$ for which we identify $e$ and $\ell$ (apply Theorem~\ref{iso} to Bl$_5\FF_3$ with the $(-3)$ curve contracted to a $\frac{1}{3}(1,1)$ singularity). 

We list all 64 pre-shrinkable surfaces in the following tables. Given a surface, we have chosen a surface that is physically equivalent to it and listed it in the first column. In case readers want to compare with \cite{JKKV}, the representative in the third column is the one in Figure 16 of \cite{JKKV}. For two pre-shrinkable surfaces, Bl$_9\FF_6\cup\PP^2$ and Bl$_{10}\FF_6\cup\PP^2$, the author does not yet know how to directly embed them. Instead, we can embed Bl$_9\FF_6\cup\FF_4$ and Bl$_8\FF_5\cup\FF_3$. These are, after HW transition and flops, physically equivalent to Bl$_9\FF_6\cup\PP^2$ and Bl$_{10}\FF_6\cup\PP^2$ respectively. However, Bl$_9\FF_6\cup\FF_4$ and Bl$_8\FF_5\cup\FF_3$ are not pre-shrinkable. It would be an interesting problem to find a way to construct more CY 3-folds covering these two cases.

To give more explanation for the contents of the tables, all surfaces in Table~\ref{fig:table1} are shrinkable. In Table~\ref{fig:table2}, in the last column, we indicate surfaces for which our results do not yield an embedding. For shrinkable surfaces, the last column is blank. In the second column, we record the class of the double curve in each component.

For $dP_2\cup dP_2$, again we can flop a rational curve to a non-pre-shrinkable surface, but indeed it is well-known that we can directly embed it using toric geometry. The following is a section (at the $z=1$ plane, in $xyz$-space) of a toric fan of a CY 3-fold containing $dP_2\cup dP_2$. The two interior vertices correspond to $dP_2$.

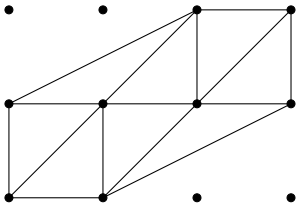
\begin{figure}[hbtp!]
\centering
\begin{tikzpicture}

\foreach \x in {0,1,2,3}
   \foreach \y in {0,1,2} 
      \draw[fill] (5/4*\x,5/4*\y) circle (1.5pt) coordinate (m-\x-\y);

\draw (m-0-0) -- (m-0-1) -- (m-2-2)  -- (m-3-2) -- (m-3-1) -- (m-1-0) -- cycle;
\draw (m-0-0) -- (m-1-1);
\draw (m-1-1) -- (m-0-1);
\draw (m-1-1) -- (m-1-0);
\draw (m-1-1) -- (m-2-1);
\draw (m-1-1) -- (m-2-2);
\draw (m-2-1) -- (m-2-2);
\draw (m-2-1) -- (m-3-1);
\draw (m-2-1) -- (m-3-2);
\draw (m-2-1) -- (m-1-0);

\end{tikzpicture}
\caption{Section of a toric fan of a Calabi--Yau 3-fold containing $dP_2\cup dP_2$}
\label{fig:blp2}
\end{figure}

\begin{table}
\begin{tabular}{ |p{3cm}||p{4cm}|p{4cm}|  }
 \hline
 \multicolumn{3}{|c|}{Rank 2 surfaces(31)} \\
 \hline
 Surface & double curve &Comments or representative in \cite{JKKV}\\
 \hline
 $\PP^2\cup \FF_3$   & $\ell\sim e$    &   \\
 $\PP^2\cup \FF_6$   & $2\ell\sim e$  & \\
 $\FF_2\cup \FF_1$ &$e\sim f$ &  \\
 $\FF_1\cup\FF_1$    &$e\sim e$ &   \\
 $\FF_2\cup\FF_0$&   $e\sim f_1$  & \\
 $\FF_3\cup\FF_1$& $e\sim e+f$   &\\
 $\FF_4\cup\FF_0$& $e\sim f_1+f_2$   &\\
 $\FF_5\cup\FF_1$ &$e\sim e+2f$ & \\
 $\FF_6\cup\FF_0$ &$e\sim f_1+2f_2$  &\\
 $\FF_7\cup\FF_1$& $e\sim e+3f$ & \\
 $\FF_8\cup\FF_0$ &$e\sim f_1+3f_2$  &\\
 $\FF_6\cup\FF_1$ &$e\sim 2e+2f$ & \\
 $\FF_1\cup dP_2$ & $e\sim \ell-x_1-x_2$  &\\
 $\FF_1\cup dP_2$ & $e\sim x_1$ &  \\
 $\FF_2\cup dP_2$ & $e\sim \ell-x_1$ & \\
 $\FF_3\cup dP_2$&$e\sim \ell$&\\
 $\FF_4\cup dP_2$ &$e\sim 2\ell-x_1-x_2$ & \\
 $\FF_5\cup dP_2$ &$e\sim 2\ell-x_1$ &\\
 $\FF_6\cup dP_2$ &$e\sim 3\ell-2x_1-x_2$ &\\
 $\FF_7\cup dP_2$ &$e\sim 3\ell-2x_1$ &\\
 $\FF_6\cup dP_2$ & $e\sim 2\ell$ &\\

$dP_2\cup dP_2$
& $\ell\sim e$ & via toric diagram \\
 Bl$_3\FF_3\cup\PP^2$ & $e\sim \ell$& Bl$_1\FF_1\cup_{\ell-x_1-x_2} dP_2$\\
 Bl$_2\FF_2\cup \FF_0$ &$e\sim f_1$ &  Bl$_1\FF_1\cup_{x_1} dP_2$ \\
 $\FF_1\cup dP_3$ &$e\sim \ell-x_1-x_2$  & \\
 $\FF_2\cup dP_3$ &$e\sim \ell-x_1$&\\
 $\FF_3\cup dP_3$ &$e\sim 2\ell-x_1-x_2-x_3$ &\\
 $\FF_4\cup dP_3$ &$e\sim 2\ell-x_1-x_2$  &\\
 $\FF_5\cup dP_3$ &$e\sim 2\ell-x_1$ & \\
 $\FF_6\cup dP_3$ &$e\sim3\ell-2x_1-x_2$&\\
  $\FF_6\cup dP_3$ &$e\sim2\ell$&\\
 \hline
\end{tabular}
 \caption{31 pre-shrinkable surfaces, all of them are shrinkable}
    \label{fig:table1}
\label{sh}
\end{table}

    \begin{table}
\begin{tabular}{ |p{2cm}||p{4.5cm}|p{3cm}|p{4cm}|  }
\hline
 \multicolumn{4}{|c|}{Rank 2 surfaces(33)} \\
 \hline
 Surface & double curve &Comments or representative in \cite{JKKV}&Shrinkablility\\
 \hline
 Bl$_4\FF_3\cup\PP^2$   & $e\sim\ell$    &Bl$_2\FF_1\cup dP_2$&   \\
 Bl$_3\FF_2\cup \FF_0$   & $e\sim f_1$  &Bl$_1\FF_1\cup dP_3$ 
 
 &\\
 $\FF_1\cup dP_4$ &$e\sim \ell-x_1-x_2$ & &  \\
 $\FF_2\cup dP_4$    &$e\sim 2\ell-x_1-x_2-x_3-x_4$ & &  \\
 $\FF_3\cup dP4$&   $e\sim 2\ell-x_1-x_2-x_3$  & &\\
 $\FF_4\cup dP_4$& $e\sim 2\ell-x_1-x_2$  &    &\\
 $\FF_5\cup dP_4$& $e\sim 2\ell-x_1$  & &\\
 Bl$_5\FF_3\cup \PP^2$ &$e\sim \ell$ &Bl$_3\FF_1\cup dP_2$ & \\
 Bl$_4\FF_3\cup \FF_1$ &$e\sim e+f$ & Bl$_2\FF_1\cup dP_3$&\\
 Bl$_4\FF_2\cup\FF_0$& $e\sim f_1$ &Bl$_1\FF_1\cup dP_4$ 
 &\\
 $\FF_1\cup dP_5$ &$e\sim 2\ell-x_1-x_2-x_3-x_4-x_5$ & &\\
 $\FF_2\cup dP_5$ &$e\sim 2\ell-x_1-x_2-x_3-x_4$ & &\\
 $\FF_3\cup dP_5$ & $e\sim \ell$ & &\\
 $\FF_4\cup dP_5$ & $e\sim 2\ell-x_1-x_2$ & & \\
 Bl$_6\FF_3\cup\PP^2$ & $e\sim \ell$ &Bl$_4\FF_1\cup dP_2$ &\\
Bl$_5\FF_3\cup \FF_1$&$e\sim e+f$&Bl$_2\FF_1\cup dP_4$& \\
 Bl$_6\FF_6\cup\PP^2$ &$e\sim 2\ell$ &Bl$_1\FF_1\cup dP_5$ &\\
 $\FF_1\cup dP_6$ &$e\sim 2\ell-x_1-x_2-x_3-x_4-x_5$ &&\\
 $\FF_2\cup dP_6$ &$e\sim \ell-x_1$ &&\\
 $\FF_3\cup dP_6$ &$e\sim \ell$ &&\\
 Bl$_7\FF_3\cup\PP^2$ & $e\sim \ell$ &Bl$_5\FF_1\cup dP_2$&\\
 Bl$_6\FF_4\cup\FF_2$ & $e\sim f_1+f_2$ & Bl$_3\FF_1\cup dP_4$&up to HW transition\\
Bl$_6\FF_3\cup\FF_1$ & $e\sim e+f$& Bl$_2\FF_1\cup dP_5$&\\
 Bl$_6\FF_2\cup\FF_0$ &$e\sim f_1$ & Bl$_1\FF_1\cup dP_6$
 &\\
 $\FF_1\cup dP_7$ &$e\sim x_1$ & & \\
 Bl$_8\FF_3\cup\PP^2$ &$e\sim \ell$&&\\
Bl$_7\FF_4\cup \FF_2$ &$e\sim f_1+f_2$ &Bl$_3\FF_1\cup dP_5$
& up to HW transition\\
Bl$_7\FF_3\cup \FF_1$  &$e\sim e+f$ &Bl$_2\FF_1\cup dP_6$
&\\
Bl$_7\FF_2\cup\FF_0$ &$e\sim f_1$ & &\\
Bl$_9\FF_6\cup\PP^2$ &$e\sim 2\ell$&Bl$_4\FF_1\cup dP_5$& not known yet\\
 Bl$_8\FF_4\cup \FF_2$  &$e\sim f_1+f_2$&Bl$_3\FF_1\cup dP_6$& up to HW transition\\
Bl$_8\FF_3\cup \FF_1$& $e\sim e+f$& &\\
Bl$_{10}\FF_6\cup \PP^2$&$e\sim 2\ell$&Bl$_4\FF_1\cup dP_6$
&not known yet\\
 \hline
\end{tabular}
    \caption{33 pre-shinkable surfaces}
    \label{fig:table2}
\end{table}

\section{Application: Local Gromov--Witten theory of singular surfaces}

The construction of a smooth CY 3-fold containing a singular surface $S$ allows us to define local Gromov--Witten invariants for $S$ with an assumption on the moduli space of stable maps. This is the subject of the paper \cite{preprint}. There we worked with an assumption that the surface is embedded in a smooth 3-fold as a hypersurface. Lemma 2.1 in \cite{preprint} is to ensure that these invariants can be defined without assuming that the ambient 3-fold is Calabi--Yau. Now that we have constructed a CY 3-fold containing $S$, the invariants can be defined for them much more easily.

\begin{prop}
    Let $S$ be a shrinkable surface and $0\neq\beta\in H_2(S,\ZZ)$ be a curve class and $X$ be a smooth CY 3-fold containing $S$ with $i:S\hookrightarrow X$. If
    \begin{equation}\label{iden}
        \overline{\mathcal{M}}_{g}(S,\beta)\subset \overline{\mathcal{M}}_g(X,i_*\beta)
    \end{equation}
    is a union of connected components, then the contribution of $S$ to the local GW invariants of $X$ can be defined as in \cite{preprint}, Theorem 3.9.
\end{prop}
The condition for this proposition holds for all curve classes when $\omega_S^\vee$ is ample. For a shrinkable surface, this is not necessarily true. Rather, this is true for \emph{some} curve classes \cite{preprint}. This restriction is consistent with physics. The curve classes that do not have local GW invariants correspond to holomorphic 2-cycles that do not generate BPS states and they are often manually removed as in the case of $(-2)$ curves on $\FF_2.$

\section{Examples: Local weighted projective planes}
Weighted projective stacks $\mathcal{P}(a_0,a_1,a_2)$ of dimension 2 provide many interesting examples. These are canonical stacks associated with the underlying weighted projective plane thought of as a variety.  

For a weighted projective stack $\mathcal{P}(a_0,\dots,a_n)$ to be an orbifold, we need to have $d=\gcd(a_0,\dots,a_n)=1$ (see for instance \cite{toricdm,wps}). If $d>1$, then there is a natural map 
\begin{equation}
    \mathcal{P}(a_0,\dots,a_n)\to \mathcal{P}\left(\frac{a_0}{d},\dots,\frac{a_n}{d}\right)
\end{equation}
which defines a $\mu_d$-gerbe over the base, which is an orbifold. A weighted projective plane $\mathcal{P}(a_0,a_1,a_2)$ can be covered by 3 local charts
\begin{equation}
    U_i\simeq[\CC^2/\mu_{a_i}]
\end{equation}
where $\CC^2$ is $\CC^3$ with one coordinate set equal to 1 with the action given by
\begin{align}
    \zeta_{a_0}\cdot(1,z_1,z_2)=(1,\zeta^{a_1}z_1,\zeta^{a_2}z_2)\\
    \zeta_{a_1}\cdot(z_0,1,z_2)=(\zeta^{a_0}z_0,1,\zeta^{a_2}z_2)\\
    \zeta_{a_2}\cdot(z_0,z_1,1)=(\zeta^{a_0}z_0,\zeta^{a_1}z_1,1)
\end{align}
where $\zeta_{a_i}=e^{\frac{2\pi i}{a_i}}\in \mu_{a_i}.$

A weighted projective stack is smooth as a stack (its coarse moduli space is often not smooth), so there is a canonical line bundle. By the local $\mathcal{P}(a_1,a_2,a_3)$, we mean the total space of the canonical bundle of $\mathcal{P}(a_1,a_2,a_3)$, which is a CY 3-orbifold. Even when the coarse moduli space of $\mathcal{P}(a_1,a_2,a_3)$ is smooth, the coarse moduli space of the local surface can have (not necessarily isolated) singularities. When resolved, the resulting (smooth and non-compact) CY 3-folds provide interesting examples. In general, singularities are not just transverse $A_n$ singularities. Therefore a crepant resolution may not be unique (for instance, one can flop after resolving singularities).

Since we are dealing with toric stacks, taking a root stack along a torus invariant divisor corresponds to extending the ray $\rho$ corresponding to that divisor. In particular, taking root stacks can be thought of as stacky blow-ups \cite{func}. Using this, we can easily draw a fan for some weighted projective orbifold $\mathcal{P}(a_0,a_1,a_2).$

\begin{ex}
    Consider the (classical) weighted projective plane $\PP(1,1,3)$. It has a unique singular point which is a cyclic quotient singularity. Locally near the singular point, it is described by
    \begin{equation}
        \omega\cdot(x,y)\mapsto (\omega x,\omega y), \omega=e^{2\pi i/3}
    \end{equation}
    which is usually called a $\frac{1}{3}(1,1)$ type singularity. The canonical stack of this surface is the weighted projective stack $\mathcal{P}(1,1,3)$. The total space of the canonical bundle is defined in the usual way, which is a CY 3-orbifold.
    
    We then consider its coarse moduli space. Over the singular point, the fiber has the induced action $\omega\cdot z\mapsto \omega z$, because the action lifts to $\omega\cdot (dx\wedge dy)=\omega^2dx\wedge dy.$ This, in particular, implies that the only singularity is the singular point in the zero section. This is a 3-fold cyclic quotient singularity $\frac{1}{3}(1,1,1)$ which has a crepant resolution with the exceptional divisor isomorphic to $\PP^2.$

    In the resulting smooth CY 3-fold $Y$, the only compact divisors are $\PP^2\cup \FF_3$ where a line $\ell$ in $\PP^2$ is glued to the $(-3)$ curve in $\FF_3$. Hence we observe that $\PP^2\cup\FF_3$ arises from `local $\mathcal{P}(1,1,3)$'. This is a special instance of Theorem~\ref{mainthm2} and has already been studied in the literature \cite{frac}.
\end{ex}

\begin{ex}
    Consider the weighted projective stack $\mathcal{S}=\mathcal{P}(1,2,2)$. This is isomorphic to a root stack $\sqrt{\ell/\PP^2}$. Again we take the total space $\mathcal{X}=$Tot$(\omega_{\mathcal{S}})$ of the canonical bundle of $\mathcal{S}$ which is an orbifold CY 3-orbifold.
    
    Consider the coarse moduli space $X$ of $\mathcal{X}$. For this, the stacky locus is of codimension 1 instead of 2, so we need to understand the action on the fiber over the stacky locus $\ell$. Now the local model is given by a reflection $(x,y)\mapsto (x,-y)$. It follows that the fiber has a nontrivial action of $(-1)$. Therefore $X$ has transverse $A_1$ singularities along $\ell$ in the zero section (isomorphic to $\PP^2$, as it is the coarse moduli space of $\mathcal{P}(1,2,2)$. It has a small resolution by blowing up the singular locus, which is automatically crepant. Resolving this singularity, again we get $\PP^2\cup \FF_3$ inside a smooth CY 3-fold $Y$. This is a special case of Theorem~\ref{main}
\end{ex}

\begin{ex}\label{f2sh}
    When considering the coarse moduli space of a local projective plane, it is possible to have singularities not contained in the zero section as expected in Lemma~\ref{iso}. Consider $\mathcal{P}(1,1,2)$. It is isomorphic to a singular quadric cone in $\PP^3$. When we consider its local surface, then its coarse moduli space has singularities along the singular point. This is because the stabilizer group $\mu_2$ of the vertex of the cone acts trivially on the fiber of the canonical bundle. Therefore, we get transverse $A_1$ singularities again, but along a non-compact curve. Resolving this singularity, we get a copy of $\FF_2\cup R$ where $R$ is a non-compact divisor. Note that $\FF_2$ is not del Pezzo, but together with this non-compact divisor, it can be contracted to $\CC^3/\mu_4$ inside a smooth CY 3-fold.

\end{ex}

\begin{ex}\label{z7}
   Consider $\mathcal{P}(1,2,4)$. It is well-known that the underlying variety is just $\PP(1,1,2)$. If we consider the toric DM stack obtained by taking root construction along $\PP(1,2)\subset \PP(1,1,2)$, we get a toric stack that has a cyclic Picard group. By \cite{toricdm}, it must be a weighted projective stack. This way we get an identification $\mathcal{P}(1,2,4)\simeq \sqrt{\mathcal{P}(1,2)/\mathcal{P}(1,1,2)}.$ Taking the total space of the canonical bundle, we get a generically transverse $A_1$ singularities along the line. Moreover, the vertex has an additional $\mu_2$ stabilizer, so this time the singularity is localized on the zero section. Looking at the local chart, it shows that the singularity has type $\frac{1}{4}(1,2).$ Note that the stacky structure is non-isolated as 2 and 4 are not relatively prime. It has generically transverse $A_1$ singularities with one dissident point which has quotient singularity type $\frac{1}{4}(1,1,2)$. Drawing a toric diagram we can see that we get $\FF_2\cup\FF_2\cup\FF_2$ as a unique compact divisor in a resolution. This surface can be contracted to a cyclic quotient singularity $\CC^3/\mu_7$ with wieght $(1,2,4)$.

 \[   \begin{tikzpicture}
\draw[step=1cm,gray,very thin] (-3.9,-3.9) grid (2.9,1.9);
\draw[thick] (0,0) -- (1,0);
\draw[thick] (0,-1) -- (1,0);
\draw[thick] (1,0) -- (0,1);
\draw[thick] (0,1) -- (0,0);
\draw[thick] (0,0) -- (-1,-2);
\draw[thick] (0,0) -- (0,-1);
\draw[thick] (-1,-2) -- (-2,-4);
\draw[thick] (-1,-2) -- (0,1);
\draw[thick] (0,-1) -- (-2,-4);
\draw[thick] (-1,-2) -- (0,-1);
\draw[thick] (0,1) -- (-2,-4) -- (1,0) -- (0,1);
\fill (0,1) circle[radius=1.5pt] node[above right] {\tiny $(0,1)$};
\fill (-1,-2) circle[radius=1.5pt] node[above left] {\tiny $(-1,-2)$};
\fill (1,0) circle[radius=1.5pt] node[below right] {\tiny $(1,0)$};
\fill (0,-1) circle[radius=1.5pt] node[below right] {\tiny $(0,-1)$};
\fill (0,0) circle[radius=1.5pt] node[below right] {\tiny $(0,0)$};
\fill (-2,-4) circle[radius=1.5pt] node[below right] {\tiny $(-2,-4)$};
\end{tikzpicture}
\]
\captionof{figure}{Section of a toric fan of local $\mathcal{P}(1,2,4)$. Three interior vertices correspond to $\FF_2.$}

\end{ex}

\begin{ex}
    Consider $\mathcal{P}(1,2,6)$. The fan for the total space of the canonical bundle of this weighted projective plane on $z=1$ plane is the following.

 \[  \begin{tikzpicture}
\draw[step=1cm,gray,very thin] (-2.9,-3.9) grid (2.9,1.9);
\draw[thick] (0,0) -- (2,0);
\draw[thick] (0,-2) -- (0,1);
\draw[thick] (0,0) -- (-1,-3);
\draw[thick] (0,-1) -- (-1,-3);
\draw[thick] (0,-1) -- (1,0);
\draw[thick] (1,0) -- (1,-1);
\draw[thick] (0,-1) -- (1,-1);
\draw[thick] (1,0) -- (0,1);
\draw[thick] (0,1) -- (-1,-3) -- (2,0) -- (0,1);
\fill (0,1) circle[radius=1.5pt] node[above right] {\tiny $(0,1)$};
\fill (2,0) circle[radius=1.5pt] node[above right] {\tiny $(2,0)$};
\fill (-1,-3) circle[radius=1.5pt] node[below right] {\tiny $(-1,-3)$};
\fill (1,0) circle[radius=1.5pt] node[below right] {\tiny $(1,0)$};
\fill (0,-1) circle[radius=1.5pt] node[below right] {\tiny $(0,-1)$};
\fill (0,0) circle[radius=1.5pt] node[below right] {\tiny $(0,0)$};
\filldraw[red] (0,-2) circle[radius=1.5pt];
\filldraw[red] (1,-1) circle[radius=1.5pt];
\end{tikzpicture}
\]
\captionof{figure}{Section of a toric fan of local $\mathcal{P}(1,2,6)$}
The red nodes represent non-compact divisors and the interior nodes represent compact divisors. Note also that this crepant resolution is not unique because we can flop one rational curve.

The coarse moduli space of the total space of the canonical bundle of $\mathcal{P}(1,2,6)$ has non-isolated singularities. It is not just transverse ADE singularities as it has a dissident point in Miles Reid's terminology. The existence of a square (with vertices (0,0),(1,0),(0,-1), and (1,-1)) tells us that we can flop a $\PP^1$, so the crepant resolution is not unique.
\end{ex}
\begin{ex}
    Toric diagram for $\mathcal{P}(1,3,4).$ The total space contains $\PP^2$ and some blow-ups of 3 points of $\FF_2$. One rational curve can be flopped to $\FF_1$ and blow-ups of two points of $\FF_1$.

 \[   \begin{tikzpicture}
\draw[step=1cm,gray,very thin] (-3.9,-3.9) grid (2.9,1.9);
\draw[thick] (0,0) -- (1,0);
\draw[thick] (0,-1) -- (1,0);
\draw[thick] (1,0) -- (0,1);
\draw[thick] (0,1) -- (0,0);
\draw[thick] (-1,-1) -- (1,0);
\draw[thick] (-1,-1) -- (-1,-2);
\draw[thick] (-1,-1) -- (-2,-3);
\draw[thick] (0,0) -- (-1,-1);
\draw[thick] (-1,-1) -- (0,1);
\draw[thick] (-1,-1) -- (-3,-4);
\draw[thick] (-1,-1) -- (0,-1);
\draw[thick] (0,1) -- (-3,-4) -- (1,0) -- (0,1);
\fill (0,1) circle[radius=1.5pt] node[above right] {\tiny $(0,1)$};
\fill (-1,-1) circle[radius=1.5pt] node[above left] {\tiny $(-1,-1)$};
\fill (-1,-2) circle[radius=1.5pt] node[below right] {\tiny $(-1,-2)$};
\fill (1,0) circle[radius=1.5pt] node[below right] {\tiny $(1,0)$};
\fill (0,-1) circle[radius=1.5pt] node[below right] {\tiny $(0,-1)$};
\fill (0,0) circle[radius=1.5pt] node[below right] {\tiny $(0,0)$};
\fill (-2,-3) circle[radius=1.5pt] node[below right] {\tiny $(-2,-3)$};
\fill (-3,-4) circle[radius=1.5pt] node[below right] {\tiny $(-3,-4)$};
\filldraw[red] (0,-1) circle[radius=1.5pt];
\filldraw[red] (-1,-2) circle[radius=1.5pt];
\filldraw[red] (-2,-3) circle[radius=1.5pt];
\end{tikzpicture}
\]

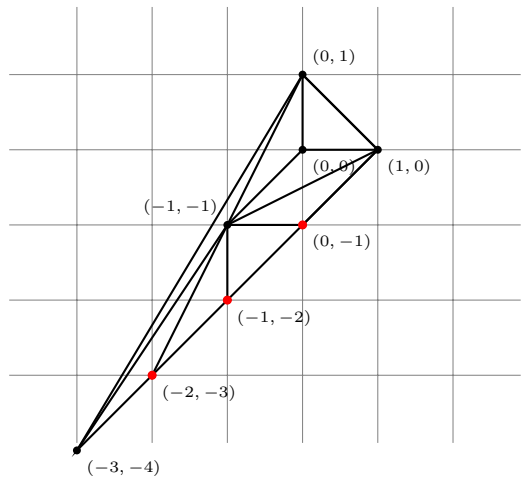
\captionof{figure}{Section of a toric fan of local $\mathcal{P}(1,3,4)$}
\end{ex}

All of these examples are isolated 3-fold quotient singularities. The following is a generalization of local $\PP^2$ viewed as a crepant resolution of $\CC^3/\mu_3.$
\begin{prop}(\cite{frac})
    Let $\CC^3/\mu_n$ be an isolated 3-fold cyclic quotient singularity where the action of $\mu_n$ is given by
    \begin{equation}
        \zeta\cdot(x,y,z)\mapsto(\zeta^a x,\zeta^by,\zeta^cz)\text{ where }\zeta=e^{2\pi i/n}
    \end{equation}
    with $a+b+c=n$, gcd$(a,b,c)$=1 and $a,b,c>0$. Then it has a partial resolution given by the local weighted projective plane $\mathcal{P}(a,b,c)$ (which is an orbifold).
\end{prop}

When $a=b=c=1$, this reduces to local $\PP^2$. Observe also that there is more than one partial resolution by a weighted projective plane. For instance, $\CC^3/\mu_5$ has two partial resolutions by $\mathcal{P}(1,2,2)$ and $\mathcal{P}(1,1,3)$. 

\begin{rmk}
    It is not true that the cyclic quotient singularity $\CC^3/\mu_n$ is determined by $n$. For instance, there are two nonequivalent quotients $\CC^3/\mu_7.$ Although the above proposition gives 4 projective planes $\mathcal{P}(1,1,5),\ \mathcal{P}(2,2,3),\ \mathcal{P}(1,3,3),\ \mathcal{P}(1,2,4)$ when $n=7$, the singularity whose partial resolution is given by $\mathcal{P}(1,2,4)$ is different from all the other three. 
\end{rmk}

In this paper we only used transverse $A_n$ singularities and isolated quotient singularities. As there are far more 3-fold canonical singularities, such as the ones discussed in this section, it would be interesting to find a systematic way of constructing other canonical singularities that comes from shrinkable surfaces.

\end{document}